\documentclass{amsart}

\usepackage{amsthm}
\usepackage[leqno]{amsmath}
\usepackage[all]{xy} \SelectTips{eu}{}
\usepackage{latexsym,amsfonts,amssymb}
\usepackage{appendix}
\usepackage{hyperref}

\newcommand{\numberseries}{\bfseries}   

\newlength{\thmtopspace}                
\newlength{\thmbotspace}                
\newlength{\thmheadspace}               
\newlength{\thmindent}                  

\newtheoremstyle{bfupright head,slanted body}
                {\thmtopspace}{\thmbotspace}
                {\slshape}{\thmindent}{\bfseries}{.}{\thmheadspace}
                {{\numberseries \thmnumber{#2\;}}\thmnote{#3}}

\newtheoremstyle{bfupright head,upright body}
                {\thmtopspace}{\thmbotspace}
                {\upshape}{\thmindent}{\bfseries}{.}{\thmheadspace}
                {{\numberseries \thmnumber{#2\;}}\thmnote{#3}}

\newtheoremstyle{fixed bf head,slanted body}
                {\thmtopspace}{\thmbotspace}{\slshape}
                {\thmindent}{\bfseries}{.}{\thmheadspace}
                {{\numberseries \thmnumber{#2\;}}\thmname{#1}\thmnote{ (#3)}}

\newtheoremstyle{fixed bf head,upright body}
                {\thmtopspace}{\thmbotspace}{\upshape}
                {\thmindent}{\bfseries}{.}{\thmheadspace}
                {{\numberseries \thmnumber{#2\;}}\thmname{#1}\thmnote{ (#3)}}

\newtheoremstyle{numbered paragraph}
                {\thmtopspace}{\thmbotspace}{\upshape}
                {\thmindent}{\upshape}{}{\thmheadspace}
                {{\numberseries \thmnumber{#2.}}}

\theoremstyle{bfupright head,slanted body}
\newtheorem{res}{}[section]             \newtheorem*{res*}{}

\theoremstyle{bfupright head,upright body}
\newtheorem{bfhpg}[res]{}               \newtheorem*{bfhpg*}{}

\theoremstyle{fixed bf head,slanted body}
\newtheorem{thm}[res]{Theorem}          \newtheorem*{thm*}{Theorem}
\newtheorem{prp}[res]{Proposition}      \newtheorem*{prp*}{Proposition}
\newtheorem{cor}[res]{Corollary}        \newtheorem*{cor*}{Corollary}
\newtheorem{lem}[res]{Lemma}            \newtheorem*{lem*}{Lemma}

\theoremstyle{fixed bf head,upright body}
\newtheorem{dfn}[res]{Definition}       \newtheorem*{dfn*}{Definition}
\newtheorem{rmk}[res]{Remark}           \newtheorem*{rmk*}{Remark}
          \newtheorem*{exa*}{Example}

\theoremstyle{numbered paragraph}
\newtheorem{ipg}[res]{}

\newlength{\thmlistleft}        
\newlength{\thmlistright}       
\newlength{\thmlistpartopsep}   
\newlength{\thmlisttopsep}      
\newlength{\thmlistparsep}      
\newlength{\thmlistitemsep}     

\newcounter{eqc} 
\newenvironment{eqc}{\begin{list}{\upshape (\textit{\roman{eqc}})}%
    {\usecounter{eqc}%
      \setlength{\leftmargin}{\thmlistleft}%
      \setlength{\labelwidth}{\thmlistleft}%
      \setlength{\rightmargin}{\thmlistright}%
      \setlength{\partopsep}{\thmlistpartopsep}%
      \setlength{\topsep}{\thmlisttopsep}%
      \setlength{\parsep}{\thmlistparsep}%
      \setlength{\itemsep}{\thmlistitemsep}}}%
  {\end{list}}%

\newcommand{\eqclbl}[1]{{\upshape(\textit{#1})}}

\newcounter{rqm}
\newenvironment{rqm}{\begin{list}{\upshape (\arabic{rqm})}%
    {\usecounter{rqm}%
      \setlength{\leftmargin}{\thmlistleft}%
      \setlength{\labelwidth}{\thmlistleft}%
      \setlength{\rightmargin}{\thmlistright}%
      \setlength{\partopsep}{\thmlistpartopsep}%
      \setlength{\topsep}{\thmlisttopsep}%
      \setlength{\parsep}{\thmlistparsep}%
      \setlength{\itemsep}{\thmlistitemsep}}}%
  {\end{list}}%

  \newcommand{\proofofimp}[3][:]{\mbox{\eqclbl{#2}$\!\implies\!$\eqclbl{#3}#1}}

\newcommand{\pgref}[1]{\ref{#1}}
\newcommand{\thmref}[2][Theorem~]{#1\pgref{thm:#2}}
\newcommand{\corref}[2][Corollary~]{#1\pgref{cor:#2}}
\newcommand{\prpref}[2][Proposition~]{#1\pgref{prp:#2}}
\newcommand{\lemref}[2][Lemma~]{#1\pgref{lem:#2}}
\newcommand{\dfnref}[2][Definition~]{#1\pgref{dfn:#2}}
\newcommand{\exaref}[2][Example~]{#1\pgref{exa:#2}}

\renewcommand{\eqref}[1]{(\pgref{eq:#1})}

\newcommand{\thmcite}[2][?]{\cite[thm.~#1]{#2}}
\newcommand{\corcite}[2][?]{\cite[cor.~#1]{#2}}
\newcommand{\prpcite}[2][?]{\cite[prop.~#1]{#2}}
 
\newcommand{\chpcite}[2][?]{\cite[chap.~#1]{#2}}
\newcommand{\seccite}[2][?]{\cite[sec.~#1]{#2}}

\def\urltilda{\kern -.15em\lower .7ex\hbox{\~{}}\kern .04em}

\newcommand{\setof}[3][\mspace{1mu}]{\{#1#2 \mid #3#1\}}
\newcommand{\kk}{\Bbbk}
\newcommand{\ZZ}{\mathbb{Z}}
\newcommand{\deq}{\:=\:}
\newcommand{\dis}{\:\is\:}
\DeclareMathOperator*{\colim}{colim}
\newcommand{\f}{\varphi}
\newcommand{\m}{\mathfrak{m}}
\newcommand{\s}{\sigma} 
\newcommand{\ot}{\otimes}
\newcommand{\is}{\cong}
\newcommand{\qis}{\simeq}
\renewcommand{\le}{\leqslant}
\renewcommand{\ge}{\geqslant}
\newcommand{\onto}{\twoheadrightarrow}
\newcommand{\lra}{\longrightarrow}
\newcommand{\xra}[2][]{\xrightarrow[#1]{\;#2\;}}
\newcommand{\qra}{\xra{\smash{\qis}}}
\newcommand{\Rop}{R^\circ}
\newcommand{\mapdef}[4][\rightarrow]{\nobreak{#2\colon #3 #1 #4}}
\newcommand{\dmapdef}[4][\lra]{\nobreak{#2\colon #3\:#1\:#4}}
\newcommand{\Ker}[1]{\nobreak{\operatorname{Ker}#1}}

\newcommand{\Coker}[1]{\nobreak{\operatorname{Coker}#1}}
\renewcommand{\H}[2][]{\operatorname{H}_{#1}(#2)}
\newcommand{\HH}[2][]{\operatorname{H}^{#1}(#2)}
\newcommand{\Thb}[2]{#2_{{\scriptscriptstyle\ge}#1}}
\newcommand{\Hom}[3][R]{\operatorname{Hom}_{#1}(#2,#3)}
\newcommand{\Ext}[4][R]{\operatorname{Ext}_{#1}^{#2}(#3,#4)}
\newcommand{\tp}[3][R]{\nobreak{#2\otimes_{#1}#3}}
\newcommand{\tpp}[3][R]{(\tp[#1]{#2}{#3})}
\newcommand{\Tor}[4][R]{\operatorname{Tor}^{#1}_{#2}(#3,#4)}

\makeatletter
\def\@nobreak@#1{\mathchoice%
  {\nobreakdef@\displaystyle\f@size{#1}}%
  {\nobreakdef@\nobreakstyle\tf@size{\firstchoice@false #1}}%
  {\nobreakdef@\nobreakstyle\sf@size{\firstchoice@false #1}}%
  {\nobreakdef@\nobreakstyle\ssf@size{\firstchoice@false #1}}%
  \check@mathfonts}%
\def\nobreakdef@#1#2#3{\hbox{{%
                    \everymath{#1}%
                    \let\f@size#2\selectfont%
                    #3}}}%
\makeatother

\def\widebardisplay#1{%
  \setbox0=\hbox{$\displaystyle #1$}
  \dimen0=\wd0%
  \advance\dimen0 by -2pt
  \vbox{%
    \nointerlineskip%
    \moveright 1pt 
    \vbox{\hrule width \dimen0}%
    \nointerlineskip%
    \kern 2pt
    \box0%
    }%
  }

\def\widebartext#1{%
  \setbox0=\hbox{$#1$}
  \dimen0=\wd0%
  \advance\dimen0 by -2pt
  \vbox{%
    \nointerlineskip%
    \moveright 1pt 
    \vbox{\hrule width \dimen0}%
    \nointerlineskip%
    \kern 1.6pt
    \box0%
    }%
  }

\def\widebarscript#1{%
  \setbox0=\hbox{$\scriptstyle #1$}
  \dimen0=\wd0%
  \advance\dimen0 by -3pt
  \vbox{%
    \nointerlineskip%
    \moveright 1.5pt 
    \vbox{\hrule width \dimen0}%
    \nointerlineskip%
    \kern .8pt
    \box0%
    }%
  }

\def\widebarscriptscript#1{%
  \setbox0=\hbox{$\scriptscriptstyle #1$}
  \dimen0=\wd0%
  \advance\dimen0 by -2pt
  \vbox{%
    \nointerlineskip%
    \moveright 1pt 
    \vbox{\hrule width \dimen0}%
    \nointerlineskip%
    \kern .6pt
    \box0%
    }%
  }

\def\widebar#1{\mathchoice%
  {\widebardisplay{#1}}%
  {\widebartext{#1}}%
  {\widebarscript{#1}}%
  {\widebarscriptscript{#1}}%
  }

\newcommand{\tha}[2]{#2^{{\scriptscriptstyle\ge}#1}}

\newcommand{\btp}[3][R]{\nobreak{#2\mathbin{\widebar{\otimes}}_{#1}#3}}
\newcommand{\btpp}[3][R]{(\nobreak{#2\mathbin{\widebar{\otimes}}_{#1}#3})}
\newcommand{\ttp}[3][R]{\nobreak{#2\mathbin{\widetilde{\otimes}}_{#1}#3}}
\newcommand{\ttpp}[3][R]{(\nobreak{#2\mathbin{\widetilde{\otimes}}_{#1}#3})}
\newcommand{\bTor}[4][R]{\widebar{\operatorname{Tor}}_{#2}^{#1}(#3,#4)}
\newcommand{\Ttor}[4][R]{\smash{\operatorname{\widehat{Tor}}}_%
  {#2}^{{#1}^{\phantom{|\mspace{-6mu}}}}(#3,#4)}
\newcommand{\Stor}[4][R]{\smash{\operatorname{\widetilde{Tor}}}_%
  {#2}^{{#1}^{\phantom{|\mspace{-6mu}}}}(#3,#4)}
\newcommand{\Ctor}[4][R]{\smash{\widecheck{\operatorname{Tor}}}_%
  {#2}^{{#1}^{\phantom{|\mspace{-6mu}}}}(#3,#4)}
\newcommand{\Text}[4][R]{\smash{\operatorname{\widehat{Ext}}%
  }_{#1}^{#2^{\phantom{|}\mspace{-6mu}}}(#3,#4)}
\newcommand{\susp}{\mathsf{\Sigma}}

\DeclareFontFamily{U}{mathx}{\hyphenchar\font45}
\DeclareFontShape{U}{mathx}{m}{n}{ <5> <6> <7> <8> <9> <10> <10.95>
  <12> <14.4> <17.28> <20.74> <24.88> mathx10 }{}
\DeclareSymbolFont{mathx}{U}{mathx}{m}{n}
\DeclareFontSubstitution{U}{mathx}{m}{n}
\DeclareMathAccent{\widecheck}{0}{mathx}{"71}
\DeclareMathAccent{\wideparen}{0}{mathx}{"75}

\newcommand{\sa}{\operatorname{S}}
\newcommand{\T}{\operatorname{T}}
\newcommand{\U}{\operatorname{U}}
\newcommand{\sfli}{\operatorname{sfli}}
\newcommand{\DD}{\operatorname{D}}
\newcommand{\iI}{\mathcal{I}}
\newcommand{\xExthat}[4]{\underset{\mathrm{\scriptscriptstyle
      #1}}{\operatorname{\widehat{Ext}}}{}^{#2}(#3,#4)}
\newcommand{\XExthat}[1]{\underset{\mathrm{\scriptscriptstyle
      #1}}{\operatorname{\widehat{Ext}}}}
\newcommand{\uHom}[2]{\underline{\operatorname{Hom}}(#1,#2)}
\newcommand{\oHom}[2]{\widebar{\operatorname{Hom}}(#1,#2)}
\newcommand{\syz}[2]{\operatorname{\Omega}_{#1}#2}

\begin{document}
\title{Complete homology over associative rings}

\author[O. Celikbas]{Olgur Celikbas} \address{University of
  Connecticut, Storrs, CT~06269, U.S.A}

\email{olgur.celikbas@uconn.edu}

\author[L.\,W. Christensen]{Lars Winther Christensen}

\address{Texas Tech University, Lubbock, TX 79409, U.S.A.}

\email{lars.w.christensen@ttu.edu}

\urladdr{http://www.math.ttu.edu/\urltilda lchriste}

\author[L. Liang]{Li Liang}

\address{Lanzhou Jiaotong University, Lanzhou 730070, China}

\email{lliangnju@gmail.com}

\author[G. Piepmeyer]{Greg Piepmeyer}

\address{Columbia Basin College, Pasco, WA~99301, U.S.A.}

\email{gpiepmeyer@columbiabasin.edu}

\thanks{This research was partly supported by NSA grant H98230-14-0140
  and Simons Foundation Collaboration Grant 281886 (L.W.C.) and by
  NSFC grant 11301240 (L.L.); it was developed during L.L.'s year-long
  visit to Texas Tech University. We gratefully acknowledge the
  hospitality of the Math Department at Texas Tech, and the support
  from the China Scholarship Council.}

\date{1 June 2015}



\begin{abstract}
  We compare two generalizations of Tate homology to the realm of
  associative rings: stable homology and the J-completion of Tor, also
  known as complete homology. For finitely generated modules, we show
  that the two theories agree over Artin algebras and over commutative
  noetherian rings that are Gorenstein, or local and complete.
\end{abstract}

\maketitle

\thispagestyle{empty} \allowdisplaybreaks

\section*{Introduction}

\noindent
Tate introduced homology and cohomology theories over finite group
algebras. In unpublished work from the 1980s---eventually given an
exposition by Goichot~\cite{FGc92}---P.\,Vogel generalized both
theories to the realm of associative rings. These generalizations are
now referred to as \emph{stable (co)homology.} Tate's theories
achieved alternate generalizations through works of Triulzi and his
adviser Mislin who introduced the \emph{J-completion} of Tor
\cite{MTr-phd} and the \emph{P-completion} of covariant Ext
\cite{GMs94}. An \emph{I-completion} of contravariant Ext was studied
by Nucinkis \cite{BEN98}. On the cohomological side, the P-completion
of Ext agrees with stable cohomology; see Kropholler's survey
\cite{PHK95} and Appendix~B.

In this paper, we investigate the homological side: Under what
conditions does stable homology
$\smash{\widetilde{\operatorname{Tor}}}$ agree with the J-completion
$\smash{\widecheck{\operatorname{Tor}}}$ of Tor?  Either theory enjoys
properties that the other may not have: Complete homology has a
universal property, and stable homology is a homological functor. The
universal property of complete homology provides a comparison map from
stable homology to complete homology. Our main theorem identifies
conditions under which stable and complete homology agree. For some of
these conditions, we do not know if they also ensure that the comparison
map is an isomorphism; in fact, we ask in \pgref{question} if the
comparison map may, in some cases, fail to detect such agreement.

\begin{res*}[Main Theorem]
  Let $R$ be an Artin algebra, or a commutative Gorenstein ring, or a
  commutative complete local ring. For every finitely generated right
  $R$-module $M$ and all $i\in\ZZ$, there are isomorphisms
  \begin{equation*}
    \Ctor{i}{M}{-} \dis \Stor{i}{M}{-}
  \end{equation*}
  of functors on the category of finitely generated left $R$-modules.
\end{res*}

\noindent
By a commutative Gorenstein ring we mean a commutative noetherian ring
whose localization at every prime ideal is a Gorenstein local ring.  A
commutative local ring is tacitly assumed to be noetherian. The
isomorphisms in the Main Theorem hold for all modules over rings with
finite \emph{sfli} invariant; see \corref{sfli}. Among these rings are
von Neumann regular rings and Iwanaga-Gorenstein rings; the latter are
noetherian rings with finite injective dimension on either side. See
Colby \cite{RRC75} and Emmanouil and Talelli \cite{IEmOTl11} for
details.

Complete homology and stable homology are defined for all modules over
any associative ring, but we do not know if they agree in that
generality. For a restricted class of modules, Iacob~\cite{AIc07} has
generalized Tate homology to the setting of associative rings. We do
show that the necessary and sufficient conditions for stable homology
to agree with Tate homology are also necessary and sufficient for
complete homology to agree with Tate homology.  However, complete and
stable homology may agree even when Tate homology is not defined; see
\exaref[]{gor} and \exaref[]{vNr}.

\section{Injective completion of covariant functors} 
\label{sec:1}

\noindent
In this paper, rings are assumed to be associative algebras over a
fixed commutative ring $\kk$. Let $R$ be a ring; we adopt the
convention that an $R$-module is a left $R$-module, and we refer to
right $R$-modules as modules over the opposite ring $\Rop$.  The
functors we consider in this paper are tacitly assumed to be functors
from the category of $R$-modules to the category of $\kk$-modules.

\begin{ipg}
  Let $\iI$ be a non-empty subset of $\ZZ$.  If $\U =
  \setof{\U_{i}}{i\in \iI}$ and $\T = \setof{\T_{i}}{i\in \iI}$ are
  families of covariant functors, then by a \emph{morphism}
  $\mapdef{\upsilon}{\U}{\T}$ we mean a family
  $\setof{\mapdef{\upsilon_i}{\U_{i}}{\T_{i}}}{i\in \iI}$ of
  morphisms.
\end{ipg}

The next definition can be found in Triulzi's thesis
\cite[6.1.1]{MTr-phd}.  It is similar to the definitions of Mislin's
of P-completion of a covariant cohomological functor \cite[2.1]{GMs94}
and Nucinkis' I-completion of a contravariant cohomological
functor~\cite[2.4]{BEN98}.

\begin{dfn}
  \label{dfn:injcompletion}
  Let $\T = \setof{\T_{i}}{i\in \iI}$ be a family of covariant
  functors. The \emph{J-completion} of $\T$ is a family
  $\widecheck{\T} = \setof{\widecheck{\T}_{i}}{i\in \iI}$ of covariant
  functors together with a morphism
  $\mapdef{\tau}{\widecheck{\T}}{\T}$ that have the following
  properties.
  \begin{rqm}
  \item One has $\widecheck{\T}_{i}(E)=0$ for every injective
    $R$-module $E$ and every $i\in \iI$.
  \item If $\U = \setof{\U_{i}}{i\in \iI}$ is a family of covariant
    functors with $\U_{i}(E) = 0$ for every injective $R$-module $E$
    and all $i\in \iI$, and if $\mapdef{\upsilon}{\U}{\T}$ is a
    morphism, then there exists a unique morphism
    $\mapdef{\s}{\U}{\widecheck{\T}}$ such that $\tau\s = \upsilon$.
  \end{rqm}
\end{dfn}

It follows from part (2) above that the J-completion of $\T$ is, if it
exists, unique up to unique isomorphism.  To discuss existence we
recall the notions of homological functors and satellites, the latter
from Cartan and Eilenberg \chpcite[III]{careil}.

\begin{ipg}
  A family $\T = \setof{\T_{i}}{i\in\ZZ}$ of covariant additive
  functors is called a \emph{homological functor} if for every exact
  sequence $0 \to N' \to N \to N'' \to 0$ of $R$-modules, there is an
  exact sequence
  \begin{equation*}
    \cdots \lra \T_{i}(N') \lra \T_{i}(N) \lra \T_{i}(N'') 
    \xra{\delta_i} \T_{i-1}(N') \lra \cdots
  \end{equation*}
  of $\kk$-modules with natural connecting homomorphisms
  $\delta_i$. For example, the family $\Tor{}{M}{-} =
  \setof{\Tor{i}{M}{-}}{i\in\ZZ}$ is such a functor for every
  $\Rop$-module $M$.
\end{ipg}

\begin{ipg}
  \label{ipg:satellite}
  Let $N$ be an $R$-module and $N \qra I$ be an injective
  resolution. For every $n\ge 0$ an $n^\mathrm{th}$ \emph{cosyzygy} of
  $N$ is $\Ker{(I^{n} \to I^{n+1})}$; up to isomorphism and injective
  summands, it is independent of the choice of injective
  resolution. If the resolution is minimal, then we use the notation
  $\Omega^nN$ for the cosyzygy $\Ker{(I^{n} \to I^{n+1})}$; these
  modules are unique up to isomorphism. Notice the isomorphism
  $\Omega^0 N \is N$.

  Let $\T$ be a covariant functor.  For $n\ge 0$ the $n^\mathrm{th}$
  \emph{right satellite} of $\T$ is a functor, denoted $\sa^n\T$,
  whose value at an $R$-module $N$ is the cokernel of the homomorphism
  $\T(I^{n-1}) \to \T(\Omega^{n}N)$.  Notice that one has $\sa^{0}\T
  \is \T$ and \mbox{$\sa^{k+1}\T(N) \is \sa^1\T(\Omega^kN)$}.
\end{ipg}

\begin{ipg}
  \label{ipg:exist}
  Following Mislin's line of proof of \thmcite[2.2]{GMs94}, Triulzi
  \prpcite[6.1.2]{MTr-phd} shows that every homological functor $\T =
  \setof{\T_{i}}{i\in\ZZ}$ has a J-completion
  $\mapdef{\tau}{\widecheck{\T}}{\T}$ with $\widecheck{\T}_{i}(N) =
  \lim_{k\ge 0}\sa^{k}\T_{k+i}(N)$ for every $R$-module $N$ and every
  $i\in\ZZ$.
\end{ipg}

\begin{rmk}
  \label{rmk:directed sytem}
  Let $\T = \setof{\T_{i}}{i\in\ZZ}$ be a homological functor. Fix $i
  \in \ZZ$; every exact sequence $0 \to \Omega^{k-1} N \to I^{k-1} \to
  \Omega^{k} N \to 0$ yields an exact sequence
  \begin{equation*}
    \T_{k+i}(I^{k-1}) \lra \T_{k+i}(\Omega^{k}N) \xra{\delta_{k+i}}
    \T_{k-1+i}(\Omega^{k-1} N)\:.
  \end{equation*}
  The connecting homomorphism $\delta_{k+i}$ induces a homomorphism
  from the cokernel $\sa^{k} \T_{k+i}(N)$ to $\T_{k-1+i}(\Omega^{k-1}
  N)$.  Composed with the natural projection onto
  $\sa^{k-1}\T_{k-1+i}(N)$ it yields a homomorphism
  $\sa^{k}\T_{k+i}(N) \to \sa^{k-1}\T_{k-1+i}(N)$.  These
  homomorphisms provide the maps in the inverse system from
  \ref{ipg:exist}.
\end{rmk}

The next result shows how to compute the J-completion directly from
cosyzygies.

\begin{lem}
  \label{lem:iso}
  Let $\T = \setof{\T_{i}}{i\in\ZZ}$ be a homological functor and let
  $N$ be an $R$-module. For every $i\in\ZZ$ there is an isomorphism
  \begin{equation*}
    \widecheck{\T}_{i}(N) \dis \lim_{k\ge 0}\T_{k+i}(\Omega^{k}N)\:.
  \end{equation*}
\end{lem}

\begin{proof}
  Let $N \qra I$ be a minimal injective resolution of $N$.  For each
  $k\ge 1$ consider the exact sequence $0 \to \Omega^{k-1}N \to
  I^{k-1} \to \Omega^{k}N \to 0$. As $\T$ is a homological functor,
  there is an exact sequence
  \begin{equation*}
    \cdots \to \T_{k+i}(I^{k-1}) \lra \T_{k+i}(\Omega^{k}N)
    \xra{\delta_{k+i}} \T_{k+i-1}(\Omega^{k-1}N) \lra
    \T_{k+i-1}(I^{k-1}) \to \cdots
  \end{equation*}
  Consider the commutative diagram
  \begin{equation*}
    \xymatrix@C=1.41pc{
      \T_{k+i}(\Omega^{k}N) \ar@{->>}[dr] \ar[rr]^-{\delta_{k+i}}
      && \T_{k+i-1}(\Omega^{k-1}N) \\
      & \sa^{k}\T_{k+i}(N)\:. \ar@{^(->}[ur]_-{\f^k}}
  \end{equation*}
  The connecting homomorphisms $\delta_{k+i}$ yield an inverse
  system. They also make up a morphism between copies of this inverse
  system, and its limit $\lim_{k\ge 0}\delta_{k+i}$ is, by the
  universal property of limits, the identity on $\lim_{k\ge
    0}\T_{k+i}(\Omega^{k}N)$. It follows that
  \begin{equation*}
    \dmapdef{\lim_{k\ge 0}\f^k}{\lim_{k\ge 0}\sa^{k}\T_{k+i}(N)}
    {\lim_{k\ge 0}\T_{k+i}(\Omega^{k}N)}
  \end{equation*} 
  is surjective, and it is injective as lim is left exact. Thus, one
  has
  \begin{equation*}
    \widecheck{\T}_{i}(N) \deq \lim_{k\ge 0}\sa^{k}\T_{k+i}(N) \dis \lim_{k\ge 0}
    \T_{k+i}(\Omega^{k}N)\:.\qedhere
  \end{equation*}
\end{proof}

It is not given that the J-completion of a homological functor is
itself a homological functor, but dimension shifting still works.

\begin{lem}
  \label{lem:dimshift}
  Let $\T = \setof{\T_{i}}{i\in\ZZ}$ be a homological functor and let
  $N$ be an $R$-module. For every $n \ge 0$, there is an isomorphism
  $\widecheck{\T}_i(\Omega^n N) \is \widecheck{\T}_{i-n}(N)$.
\end{lem}

\begin{proof} 
  For $n=0$ the isomorphism is trivial, and by induction it is
  sufficient to prove it for $n=1$. The isomorphisms in the next
  computation hold by \lemref{iso}:
  \begin{align*}
    \widecheck{\T}_i(\Omega N) 
    & \dis \lim_{k\ge 0} \T_{k+i}(\Omega^k(\Omega N))\\
    & \deq \lim_{k \ge 0} \T_{(k + 1) + (i - 1)}(\Omega^{k + 1} N) \\
    & \deq \lim_{k \ge 1} \T_{k + (i - 1)}(\Omega^k N) \\
    & \dis \widecheck{\T}_{i-1}(N)\:.\qedhere
  \end{align*}
\end{proof}

\begin{lem}
  \label{lem:induced isomorphism}
  Let $\mapdef[\lra]{\upsilon}{\setof{\U_{i}}{i\in\ZZ}}%
  {\setof{\T_{i}}{i\in\ZZ}}$ be a morphism of homological functors
  such that $\upsilon_i$ is an isomorphism for $i \gg 0$. If\,
  $\U_{i}(E)=0$ holds for each injective $R$-module $E$ and every
  $i\in\ZZ$, then the unique morphism that exists by
  \dfnref{injcompletion}, $\mapdef[\lra]{\s}{\setof{\U_{i}}{i\in\ZZ}}%
  {\setof{\widecheck{\T}_{i}}{i\in\ZZ}}$, is an isomorphism.
\end{lem}

\begin{proof}
  Set $\iI = \setof{i \in \ZZ}{\upsilon_i\colon U_i \to \T_i\ \text{is
      an isomorphism}}$.  By assumption $\iI$ contains all integers
  sufficiently large. Since $\U $ vanishes on injective modules, it
  follows that $\setof{\U_i}{i \in \iI}$ is the J-completion of
  $\setof{\T_i}{i \in \iI }$, i.e., $\U_i \is \widecheck{\T}_i$ for
  all $i \in \iI$; see Definition \ref{dfn:injcompletion}.

  Now let $j\in\ZZ$, choose an $i \in\iI$ with $i \ge j$, and set
  $n=i-j$. For every $R$-module $N$ there are isomorphisms
  \begin{align*}
    \widecheck{\T}_j(N) \deq \widecheck{\T}_{i - n}(N)
    &\dis \widecheck{\T}_i(\Omega^n N)\\
    &\dis \U_i(\Omega^n N)\\
    &\dis \U_{i-n}(N) \deq \U_{j}(N)\:.
  \end{align*}
  The first isomorphism holds by \lemref{dimshift}, and the second
  holds by the argument above as $i$ is in $\iI$. The last isomorphism
  follows by dimension shifting, as $\U$ is a homological functor that
  vanishes on injective modules.
\end{proof}

\section{Comparison to stable homology} 
\label{sec:2}

We now focus on the J-completion of the homological functor
$\Tor{}{M}{-}$.

\begin{dfn}
  \label{dfn:complete}
  Let $M$ be an $\Rop$-module; for the J-completion of the homological
  functor $\Tor{}{M}{-} = \setof{\Tor{i}{M}{-}}{i\in\ZZ}$, see
  \dfnref{injcompletion}, we use the notation $\Ctor{}{M}{-} =
  \setof{\Ctor{i}{M}{-}}{i\in\ZZ}$. For every $R$-module $N$, the
  $\ZZ$-indexed family of $\kk$-modules $\Ctor{i}{M}{N}$ is called the
  \emph{complete homology} of $M$ and $N$ over $R$.
\end{dfn}

\begin{ipg}
  \label{ipg:vanishing}
  If $M$ has finite flat dimension, or if $N$ has finite injective
  dimension, then the complete homology $\Ctor{}{M}{N}$ vanishes: this
  follows straight from the isomorphisms $\Ctor{i}{M}{N} \is
  \lim_{k\ge 0} \Tor{k+i}{M}{\Omega^{k}N}$ from \lemref{iso}.
\end{ipg}

The goal of this section is to compare complete homology to stable
homology; we recall the definition of stable homology and refer to
\seccite[2]{CCLP} or \cite{FGc92} for details.

\begin{ipg}
  \label{ipg:complete}
  Let $M$ be an $\Rop$-module and $N$ be an $R$-module. Let $P \qra M$
  be a projective resolution and $N \qra I$ be an injective
  resolution. The tensor product complex $\tp{P}{I}$ has components
  $\tpp{P}{I}_n = \coprod_{i\in\ZZ}\tp{P_{i+n}}{I^{i}}$; it is a
  subcomplex of $\btp{P}{I}$ with components $\btpp{P}{I}_n =
  \prod_{i\in\ZZ}\tp{P_{i+n}}{I^{i}}$ and differential extended by
  linearity from $p\ot i \mapsto \partial(p)\ot i +
  (-1)^{|p|}p\ot\partial(i)$, where $|p|$ is the (homological) degree
  of $p$. The quotient complex is denoted $\ttp{P}{I}$, and the
  modules
  \begin{equation*}
    \Stor[R]{i}{M}{N} \deq \H[i+1]{\ttp{P}{I}}
  \end{equation*}
  make up the \emph{stable homology} of $M$ and $N$ over $R$. They fit
  in a long exact sequence with homology modules $\Tor{i}{M}{N}$ and
  $\bTor{i}{M}{N} = \H[i]{\btp{P}{I}}$; see~\cite[2.5]{CCLP}.
\end{ipg}

\subsection*{\em Comparison via the universal property}
Let $M$ be an $\Rop$-module. Stable homology $\Stor{}{M}{-}$ is a
homological functor, and there is a morphism
\begin{equation*}
  \dmapdef{\eth}{\Stor{}{M}{-}} {\Tor{}{M}{-}}
\end{equation*}
given by connecting maps in a long exact sequence; see
\cite[2.5]{CCLP}.  Stable homology $\Stor{}{M}{E}$ vanishes for every
injective $R$-module $E$; see \cite[2.3]{CCLP}. By the universal
property of complete homology, there is thus a morphism
\begin{equation*}
  \dmapdef{\s}{\Stor{}{M}{-}}{\Ctor{}{M}{-}} 
\end{equation*}
with $\tau\s = \eth$ where
$\mapdef{\tau}{\Ctor{}{M}{-}}{\Tor{}{M}{-}}$ is as in
\dfnref{injcompletion}.

We aim for an explicit description of the morphisms $\tau$, $\eth$,
and $\s$. To that end we first reinterpret the computation of complete
homology in \lemref{iso}.

\begin{ipg}
  \label{ipg:lim}
  Fix $i\in\ZZ$. Recall from the proof of \lemref{iso} that the maps
  in the inverse system on the right-hand side of $\Ctor{i}{M}{N} \is
  \lim_{k\ge 0}\Tor{k+i}{M}{\Omega^{k}N}$ come from the connecting
  homomorphisms
  \begin{equation}
    \label{eq:0}
    \dmapdef{\delta_{k+i}}{\Tor{k+i}{M}{\Omega^{k}N}}{\Tor{k+i-1}{M}{\Omega^{k-1}N}}\:.
  \end{equation}
  Let $P \qra M$ be a projective resolution and let $N\qra I$ be a
  minimal injective resolution. There is an exact sequence $0 \to
  I^{\ge k} \to I^{\ge k-1} \to \susp^{-(k-1)}I^{k-1}\to 0$, for every
  $k\ge 1$, which induces an exact sequence
  \begin{equation}
    \label{eq:1}
    0 \lra
    \tp{P}{\susp^{k}I^{\ge k}} \lra \tp{P}{\susp^{k}I^{\ge k-1}} \lra
    \tp{P}{\susp I^{k-1}} \lra 0\:.
  \end{equation}
  For every $n\ge 0$ the canonical map $\Omega^nN \to \susp^nI^{\ge
    n}$ is a minimal injective resolution. It follows that the exact
  sequence in homology associated to \eqref{1} yields homomorphisms
  from
  \begin{alignat*}{2}
    \H[k+i]{\tp{P}{\susp^{k}I^{\ge k}}} &\deq \Tor{k+i}{M}{\Omega^{k}N}\qquad\text{to}\\
    \H[k+i]{\tp{P}{\susp^{k}I^{\ge k-1}}} &\deq
    \H[k+i-1]{\tp{P}{\susp^{k-1}I^{\ge k-1}}}
    =\Tor{k+i-1}{M}{\Omega^{k-1}N}\:.
  \end{alignat*}
  That these maps yield an inverse system isomorphic to the one given
  by the maps \eqref{0} is due to the diagram
  \begin{equation*}
    \xymatrix{
      \H[k+i]{\tp{P}{\susp^k\tha{k}{I}}} \ar[r] \ar[d] ^-\is
      & \H[k+i-1]{\tp{P}{\susp^{k-1}\tha{k-1}{I}}} \ar[d]^-\is \\
      \Tor{k+i}{M}{\Omega^{k}N} \ar[r]^-{\delta_{k+i}}
      & \Tor{k+i-1}{M}{\Omega^{k-1}N}
    }
  \end{equation*}
  where the top horizontal map is induced by the embedding $I^{\ge k}
  \to I^{\ge k-1}$.  The diagram is commutative because $\tp{P}{-}$ is
  a triangulated functor on the derived category, and $\H{-}$ is a
  homological functor on the derived category. This explains the
  second isomorphism in the chain
  \begin{equation}
    \label{eq:2}
    \begin{split}
      \Ctor{i}{M}{N} 
      &\dis \lim_{k\ge 0}\Tor{k+i}{M}{\Omega^{k}N} \\
      &\dis \lim_{k\ge 0}\H[k+i]{\tp{P}{\susp^{k}\tha{k}{I}}} \dis
      \lim_{k\ge 0}\H[i]{\tp{P}{\tha{k}{I}}}
    \end{split}
  \end{equation}
  where the limits in the last line are taken over the system given by
  the maps
  \begin{equation}
    \label{eq:3}
    \dmapdef{\iota^k}{\H[i]{\tp{P}{\tha{k}{I}}}}{\H[i]{\tp{P}{\tha{k-1}{I}}}}
  \end{equation}
  induced by the embeddings $\tha{k}{I} \to \tha{k-1}{I}$.
\end{ipg}

\begin{ipg}
  \label{ipg:tes}
  Adopt the notation from \ref{ipg:lim}. In view of
  \ref{ipg:lim}\eqref{2}, an element in complete homology
  $\Ctor{i}{M}{N}$ is a sequence of homology classes $([w^{\ge k}])_{k
    \ge 0}$ where each $w^{\ge k}$ is a cycle in $\tpp{P}{I^{\ge
      k}}_i$, and one has $\iota^k([w^{\ge k}]) = [w^{\ge k-1}]$. With
  this notation,
  \begin{equation*}
    \dmapdef{\tau_i}{\Ctor{i}{M}{N}}{\Tor{i}{M}{N}}\quad\text{is
      given by}\quad ([w^{\ge k}])_{k \ge 0} \longmapsto [w^{\ge 0}]\:;
  \end{equation*}
  cf.~\cite[proof of 1.2.2 and text after 6.1.2]{MTr-phd}.

  A homogeneous cycle $z$ in $\ttp{P}{I}$ of homological degree $|z|$
  is represented by a family $(z^j)_{j\ge 0}$ with $z^j \in
  \tp{P_{j+|z|}}{I^j}$, such that $\partial(z)$ belongs to
  $\tpp{P}{I}_{|z|-1}$. That is, $z$ may have infinite support and
  need not be a cycle in $\btp{P}{I}$, but $\partial(z)$ has finite
  support, and by the definition of the differential, $\partial(z)$ is
  a cycle in $\tp{P}{I}$.  Thus, an element in $\Stor{i}{M}{N}$ is the
  class $[z]$ of a cycle in $\ttpp{P}{I}_{i+1}$ with $\partial(z)$ a
  cycle in $\tpp{P}{I}_i$, and with this notation the connecting
  homomorphism
  \begin{equation*}
    \dmapdef{\eth_i}{\Stor{i}{M}{N}}{\Tor{i}{M}{N}}\quad\text{is
      given by}\quad [z] \longmapsto [\partial(z)]\:.
  \end{equation*}

  Let $z=(z^j)_{j\ge 0}$ be a cycle in $\ttpp{P}{I}_{i+1}$; for every
  $k\ge 0$ set $z^{\ge k} = (z^j)_{j\ge k}$. We just saw that
  $\partial(z)$ has finite support, so for every $k\ge 0$ the boundary
  $\partial(z^{\ge k})$ is a cycle in $\tpp{P}{I^{\ge k}}_i$.  In
  $\tp{P}{I^{\ge k-1}}$ the difference $\partial(z^{\ge k-1})
  - \partial(z^{\ge k}) = \partial(z^{k-1})$ is a boundary, so the
  sequence of homology classes $([\partial(z^{\ge k})])_{k\ge 0}$ is
  compatible with the morphisms $\iota^k$ from
  \ref{ipg:lim}\eqref{3}. To see that the homomorphism
  \begin{equation*}
    \dmapdef{\s_i}{\Stor{i}{M}{N}}{\Ctor{i}{M}{N}}\quad\text{given by}
    \quad [z] \longmapsto ([\partial(z^{\ge k})])_{k\ge 0}
  \end{equation*}
  provides the desired factorization, note that for $[z]$ in stable
  homology $\Stor{i}{M}{N}$ one has $\tau_i\s_i([z]) =
  \tau_i(([\partial(z^{\ge k})])_{k\ge 0}) = [\partial(z^{\ge 0})] =
  [\partial(z)] = \eth_i([z])$.
\end{ipg}

The comparison map $\s$ is always surjective. Triulzi
\corcite[6.2.10]{MTr-phd} proves this in the case of group algebras,
and Russell \prpcite[58]{JRs-phd} has the general case. As neither
thesis is publicly available, we include a proof of surjectivity in
Appendix~A.

\begin{ipg}
  \label{ipg:fcfd}
  An $\Rop$-module $M$ is said to have \emph{finite copure flat
    dimension} if there exists an integer $n$ such that $\Tor{i}{M}{E}
  = 0$ for every injective $R$-module $E$ and all $i \ge n$; see
  Enochs and Jenda \cite{EEnOJn93a}.  Prominent examples of such
  modules are those of finite Gorenstein flat dimension; see
  \thmcite[4.14]{CFH-11}.
\end{ipg}

\begin{prp}
  \label{prp:copure}
  For an $\Rop$-module $M$ of finite copure flat dimension and for all
  $i\in\ZZ$ there are isomorphisms of functors on the category of
  $R$-modules
  \begin{equation*}
    \Ctor{i}{M}{-} \dis \Stor{i}{M}{-}\:.
  \end{equation*}
\end{prp}

\begin{proof}
  By assumption there is an integer $n$ such that $\Tor{i}{M}{E}=0$
  for every injective $R$-module $E$ and all $i \ge n$. Therefore
  \prpcite[2.9]{CCLP} yields $\Stor{i}{M}{-} \is \Tor{i}{M}{-}$ for
  all $i \ge n$. Now \cite[2.3]{CCLP} and \lemref{induced isomorphism}
  finish the proof.
\end{proof}

\begin{rmk}
  Stable homology $\Stor{}{M}{-}$ is a homological functor; see
  \cite[(2.4.2)]{CCLP}. Thus, for an $\Rop$-module $M$ of finite
  copure flat dimension, the Proposition establishes $\Ctor{}{M}{-}$
  as a homological functor.  In the proof, the isomorphisms one gets
  from \prpcite[2.9]{CCLP} are $\eth_{i}$ for $i\ge n$, hence the
  isomorphisms in the statement of \prpref{copure} are, in fact, the
  maps $\s_i$ discussed in \ref{ipg:tes}; cf.~\dfnref{injcompletion}.
\end{rmk}

\begin{ipg}
  \label{ipg:Gor result}
  Finitely generated modules over commutative Gorenstein rings have
  finite Gorenstein flat dimension: this is a result due to Goto; see
  \thmcite[1.18, prop.~4.24]{CFH-11}.  Thus Proposition
  \ref{prp:copure} establishes the Main Theorem in the case of
  Gorenstein rings.
\end{ipg}

\begin{ipg}
  \label{ipg:sfli}
  The invariant $\sfli R$ is the supremum of the flat dimensions of
  all injective $R$-modules.  Iwanaga-Gorenstein rings have finite
  sfli by \thmcite[9.1.7]{rha}.

  Over a ring $R$ with finite sfli, all $\Rop$-modules have finite
  copure flat dimension. Hence the next result is immediate from
  Proposition \ref{prp:copure}.
\end{ipg}

\begin{cor}
  \label{cor:sfli}
  Let $R$ be a ring with $\sfli R$ finite. For every $\Rop$-module $M$
  and all $i\in\ZZ$, there are isomorphisms $\Ctor{i}{M}{-} \is
  \Stor{i}{M}{-}$ of functors on the category of $R$-modules.\qed
\end{cor}

\begin{cor}
  \label{cor:noetherian}
  Let $R$ be a right Noetherian ring and let $M$ be a finitely
  generated $\Rop$-module. If\, $\Ext[\Rop]{i}{M}{R}=0$ holds for all
  $i\gg 0$, then for all $i\in\ZZ$ there are isomorphisms
  $\Ctor{i}{M}{-} \is \Stor{i}{M}{-}$ of functors on the category of
  $R$-modules.
\end{cor}

\begin{proof}
  Let $E$ be an injective $R$-module. Since $\Ext[\Rop]{i}{M}{R}=0$
  holds for all $i\gg 0$, the isomorphisms $\Tor{i}{M}{E} \is
  \Hom{\Ext[\Rop]{i}{M}{R}}{E}$ from \prpcite[VI.5.3]{careil} show
  that $M$ has finite copure flat dimension; see \ref{ipg:fcfd}. Thus
  the claim follows from \prpref{copure}.
\end{proof}

\subsection*{\em Comparison via duality}
With an exact functor $\Hom[]{-}{E}$ one can, loosely speaking, toggle
between Ext and Tor; the proof of \corref{noetherian} exemplifies
this.  To establish the Main Theorem for Artin algebras and
commutative complete local rings, we employ a duality $\Hom[]{-}{E}$
and use that the P-completion of covariant Ext agrees with stable
cohomology; see Appendix B for details.

\begin{thm}
  \label{thm:duality}
  Let $M$ be an $\Rop$-module with a degree-wise finitely generated
  projective resolution and let $E$ be an injective $\kk$-module. For
  all $i\in\ZZ$ there are isomorphisms of functors on the category of
  $\Rop$-modules
  \begin{equation*}
    \Ctor{i}{M}{\Hom[\kk]{-}{E}} \dis \Stor{i}{M}{\Hom[\kk]{-}{E}}\:.
  \end{equation*}
\end{thm}

\begin{proof}
  As the P-completion of covariant Ext agrees with stable cohomology,
  \thmcite[A.7]{CCLP} yields isomorphisms,
  \begin{equation*}
    \Stor{i}{M}{\Hom[\kk]{-}{E}} \dis \Hom[\kk]{\Text[\Rop]{i}{M}{-}}{E}\:.
  \end{equation*}
  Fix a degree-wise finitely generated projective resolution $P \qra
  M$. Let $N$ be a an $\Rop$-module and $Q \qra N$ be a projective
  resolution. The quasi-isomorphism $\Hom[\kk]{N}{E} \qra
  \Hom[\kk]{Q}{E}$ is an injective resolution over $R$, and
  $\Hom[\kk]{Q}{E}$ splits as a direct sum $I \oplus J$, where $I$
  provides a minimal injective resolution of $\Hom[\kk]{N}{E}$ and $J$
  is acyclic. Now one has:
  \begin{align*}
    \Hom[\kk]{\Text[\Rop]{i}{M}{N}}{E} 
    &\dis \Hom[\kk]{\colim_{k\ge 0}\HH[i]{\Hom{P}{\Thb{k}{Q}}}}{E} \\
    &\dis \lim_{k\ge 0}\Hom[\kk]{\HH[i]{\Hom{P}{\Thb{k}{Q}}}}{E} \\
    &\dis \lim_{k\ge 0}\H[i]{\Hom[\kk]{\Hom{P}{\Thb{k}{Q}}}{E}} \\
    &\dis \lim_{k\ge 0}\H[i]{\Hom[\kk]{\Hom{P}{\susp^{k}\Omega_{k}N}}{E}} \\
    &\dis \lim_{k\ge 0}\H[i]{\tp{P}{\Hom[\kk]{\susp^{k}\Omega_{k}N}{E}}} \\
    &\dis \lim_{k\ge 0}\H[i]{\tp{P}{\Hom[\kk]{\Thb{k}{Q}}{E}}} \\
    &\dis \lim_{k\ge 0}\H[i]{\tp{P}{\tha{k}{(I\oplus J)}}} \\
    &\dis \Ctor{i}{M}{\Hom[\kk]{N}{E}}\:.
  \end{align*}
  The first isomorphism follows from \lemref{complete cohomology}. The
  fourth and the sixth isomorphisms hold since $\Thb{k}{Q} \qra
  \susp^{k}\Omega_{k}N$ is a projective resolution. The fifth
  isomorphism holds because $P$ is degree-wise finitely generated; see
  \prpcite[VI.5.2]{careil}. The last isomorphism follows from
  \ref{ipg:lim}(3) because the complex $\tp{P}{J}$ is acyclic. As the
  first isomorphism is natural in the $(N/Q)$ variable, and the
  subsequent isomorphisms are natural in $Q$, it follows that the
  total isomorphism is natural in $N$.
\end{proof}

The next corollary accounts for the case of Artin algebras in the Main
Theorem.

\begin{cor}
  \label{cor:artin}
  Let $R$ be an Artin algebra. For every finitely generated
  $\Rop$-module $M$ and all $i\in\ZZ$, there are isomorphisms
  $\Ctor{i}{M}{-} \dis \Stor{i}{M}{-}$ of functors on the category of
  finitely generated $R$-modules.
\end{cor}

\begin{proof}
  Let $\kk$ be artinian and $R$ be finitely generated as a
  $\kk$-module. Let $E$ denote the injective hull of
  $\kk/\operatorname{Jac}\kk$ and let $\DD(-) = \Hom[\kk]{-}{E}$ be
  the duality functor for $R$. By \thmref{duality} there are
  isomorphisms of functors on the category of finitely generated
  $R$-modules,
  \begin{equation*}
    \Ctor{i}{M}{-} \dis \Ctor{i}{M}{\DD(\DD(-))} \dis 
    \Stor{i}{M}{\DD(\DD(-))} \dis \Stor{i}{M}{-}\:. \qedhere 
  \end{equation*}
\end{proof}

\begin{ipg}
  \label{ipg:matlis}
  Let $R$ be a commutative local ring and let $E$ denote the injective
  hull of its residue field. An $R$-module $N$ is called \emph{Matlis
    reflexive} if the canonical map $N \to \Hom{\Hom{N}{E}}{E}$ is an
  isomorphism. For example, finitely generated complete modules, in
  particular modules of finite length, are Matlis reflexive. With
  $\DD(-) = \Hom{-}{E}$ the isomorphisms in the proof of
  \corref{artin} remain valid and yield, for a finitely generated
  $R$-module $M$, isomorphisms
  \begin{equation*}
    \Ctor{i}{M}{-} \dis \Stor{i}{M}{-}
  \end{equation*}
  of functors on the subcategory of Matlis reflexive $R$-modules.

  Over a commutative complete local ring, Matlis duality establishes
  that all finitely generated modules and all artinian modules are
  Matlis reflexive. The next corollary accounts for the case of
  complete local rings in the Main Theorem.
\end{ipg}

\begin{cor}
  \label{cor:local}
  \protect\pushQED{\qed}%
  Let $R$ be a commutative noetherian complete local ring. For all
  finitely generated $R$-modules $M$ and all $i\in\ZZ$, there are
  isomorphisms of functors on the category of finitely generated
  $R$-modules $\Ctor{i}{M}{-} \dis \Stor{i}{M}{-}$.\qed
\end{cor}

\begin{rmk}
  \label{rmk:yoshino}
  For non-negative integers $i$, the isomorphisms in
  \corref[Corollaries~]{noetherian} and \corref[]{artin} were proved
  by Yoshino \cite[thms.~5.3 and 5.5]{YYs01}. Indeed, Yoshino
  introduces and studies in \emph{loc.\,cit.}\ the \emph{Tate--Vogel
    completion} of covariant half exact functors.  For such a functor
  $\T$---our interest is in $\Tor{i}{M}{-}$ for an $\Rop$-module
  $M$---the Tate--Vogel completion $\T^\wedge$ is given by
  $\T^\wedge(N) = \lim_{k \geq 0} \sa^k \sa_k \T(N)$, for every
  $R$-module $N$.  Here $\sa_k\T$ denotes the $k^\mathrm{th}$
  \emph{left satellite} of the functor $\T$; see
  \chpcite[III]{careil}.

  For $i\in \ZZ$ set $\T_i = \Tor{i}{M}{-}$. As $\T_i$ is $0$ for
  $i<0$, Yoshino only considers the Tate--Vogel completion
  $\T_i^\wedge$ for $i\ge 0$. However, for $i\ge 0$ dimension shifting
  yields $\sa_k\T_i \is \T_{k+i}$.  Thus, in the definition of the
  Tate--Vogel completion of $\T_i$ we may replace $\sa_k\T_i$ by
  $\T_{k+i}$, so that in the limit one gets a Tate--Vogel completion
  $\T^\wedge_i$ of $\T_i$ for all $i\in\ZZ$. By \ref{ipg:exist} and
  \dfnref{complete} this is complete homology $\Ctor{i}{M}{-}$.
\end{rmk}

\begin{rmk}
  \label{rmk:russell}
  In his thesis \cite{JRs-phd}, Russell studies an \emph{asymptotic
    stabilization} of the tensor product and compares it to Yoshino's
  work~\cite{YYs01} discussed above, to stable homology (which he calls Vogel
  homology), and to complete homology (which, not having access to
  \cite{MTr-phd}, he referres to as ``mirror-Mislin''). As already
  mentioned, Russel \prpcite[58]{JRs-phd} proves surjectivity of the
  comparison map from stable to complete homology; furthermore
  \prpcite[57]{JRs-phd} coincides with \corref{noetherian}.
\end{rmk}

\section{Comparison to Tate homology} 

\noindent
Tate (co)homology was originally defined for modules over finite group
algebras. Through works of Iacob \cite{AIc07} and Veliche~\cite{OVl06}
the theories have been generalized to the extent that one can talk
about Tate homology $\Ttor{}{M}{N}$ and Tate cohomology
$\Text{}{M}{N}$ for modules over any ring, provided that the first
argument, $M$, has a complete projective resolution; see
\pgref{ipg:gproj}.

Stable cohomology and the P-completion of covariant Ext always agree,
and they coincide with Tate cohomology whenever the latter is defined;
see Appendix~B. The questions of when stable homology and complete
homology (the J-completion of Tor) agree, and when they coincide with
Tate homology, do not yet have satisfactory answers. In \cite{CCLP} we
prove that stable homology, $\widetilde{\mathrm{Tor}}$, agrees with
Tate homology over a ring $R$ if and only if every Gorenstein
projective $\Rop$-module is Gorenstein flat; see \pgref{ipg:gproj} and
\pgref{ipg:gflat}. In this section we prove that complete homology
$\widecheck{\mathrm{Tor}}$ agrees with Tate homology under the exact
same condition. We give two proofs of this fact; the first one uses
the result for stable homology, and the second is independent thereof.

\begin{ipg}
  \label{ipg:gproj}
  An acyclic complex $T$ of projective $\Rop$-modules is called
  \emph{totally acyclic} if $\Hom[\Rop]{T}{P}$ is acyclic for every
  projective $\Rop$-module $P$.  An $\Rop$-module $G$ is called
  \emph{Gorenstein projective} if there exists a totally acyclic
  complex $T$ of projective $\Rop$-modules with $\Coker{(T_1 \to T_0)}
  \is G$.  A \emph{complete projective resolution} of an $\Rop$-module
  $M$ is a diagram $T \xra{\varpi} P \qra M$, where $T$ is a totally
  acyclic complex of projective $\Rop$-modules, $P\qra M$ is a
  projective resolution, and $\varpi_i$ is an isomorphism for $i \gg
  0$; see \seccite[2]{OVl06}. A module has a complete projective
  resolution if and only if it has finite Gorenstein projective
  dimension; see \thmcite[3.4]{OVl06}.

  Let $M$ be an $\Rop$-module with a complete projective resolution $T
  \to P \to M$, and let $N$ be an $R$-module. The \emph{Tate homology}
  of $M$ and $N$ over $R$ is the collection of $\kk$-modules
  $\Ttor{i}{M}{N}= \H[i]{\tp{T}{N}}$ for $i\in\ZZ$; see
  \seccite[2]{AIc07}.
\end{ipg}

If $R$ is noetherian and $M$ is a finitely generated $\Rop$-module
that has a complete projective resolution, then one has $\Stor{}{M}{-}
\is \Ttor{}{M}{-}$; see \thmcite[6.4]{CCLP}. Next comes the analogous
statement for complete homology.

\begin{prp}
  \label{prp:complete-Tate-fg}
  Let $R$ be a noetherian ring and let $M$ be a finitely generated
  $\Rop$-module that has a complete projective resolution.  For all
  $i\in\ZZ$, there are isomorphisms of functors on the category of
  $R$-modules,
  \begin{equation*}
    \Ctor{i}{M}{-} \dis \Ttor{i}{M}{-}\:.
  \end{equation*} 
\end{prp}

\begin{proof}
  It follows from \cite[2.4.1, 3.4]{OVl06} that
  $\Ext[\Rop]{i}{M}{R}=0$ holds for all $i\gg 0$. Therefore, by
  \corref{noetherian}, there are isomorphisms $\Ctor{i}{M}{-} \is
  \Stor{i}{M}{-}$ of functors for all $i\in\ZZ$. Now
  \thmcite[6.4]{CCLP} completes the proof.
\end{proof}

With the next observation we can give an alternate proof of
\prpref{complete-Tate-fg}, one that does not rely upon knowing how
stable homology compares to Tate homology.

\begin{ipg}
  \label{ipg:reprove}
  Let $M$ be an $\Rop$-module with a complete projective resolution
  \mbox{$T \to P \to M$}.  The morphism $T \to P$ induces a morphism
  $\upsilon$ of homological functors, whose components
  $\mapdef{\upsilon_i}{\Ttor{i}{M}{-}}{\Tor{i}{M}{-}}$ are
  isomorphisms for all $i\gg 0$. To prove that Tate homology
  $\Ttor{}{M}{-}$ agrees with complete homology $\Ctor{}{M}{-}$, it
  suffices by \lemref{induced isomorphism} to show that
  $\Ttor{}{M}{E}$ vanishes for all injective $R$-modules~$E$.
\end{ipg}

\begin{proof}[Alternate proof of \pgref{prp:complete-Tate-fg}]
  It follows from \cite[2.4.1]{OVl06} and \thmcite[3.1]{LLAAMr02} that
  $M$ has a complete projective resolution $T \to P \to M$ with $T$
  and $P$ degree-wise finitely generated. Let $E$ be an injective
  $R$-module; the complex
  \begin{equation*}
    \tp{T}{E} \dis \tp{T}{\Hom{R}{E}} \dis \Hom{\Hom[\Rop]{T}{R}}{E}
  \end{equation*}
  is acyclic; the last isomorphism is
  \cite[prop.~VI.5.2]{careil}. Thus one has
  $\Ttor{i}{M}{E}=\H[i]{\tp{T}{E}}=0$ for all $i\in\ZZ$, and now
  \ref{ipg:reprove} finishes the argument.
\end{proof}

\begin{ipg}
  \label{ipg:gflat}
  An $\Rop$-module $G$ is called \emph{Gorenstein flat} if there
  exists an acyclic complex $T$ of flat $\Rop$-modules with
  $\Coker{(T_1 \to T_0)}\is G$ and such that $\tp{T}{E}$ is acyclic
  for every injective $R$-module $E$. In general, it is not known
  whether or not Gorenstein projective modules are Gorenstein flat;
  see also Emmanouil~\thmcite[2.2]{IEm12}.
\end{ipg}

We show in \thmcite[6.7]{CCLP} that one has $\Stor{}{M}{-} \is
\Ttor{}{M}{-}$ for every $\Rop$-module $M$ with a complete projective
resolution if and only if every Gorenstein projective $\Rop$-module is
Gorenstein flat. Here is the result for complete homology.

\begin{prp}
  \label{prp:complete-Tate}
  The following conditions on $R$ are equivalent.
  \begin{eqc}
  \item Every Gorenstein projective $\Rop$-module is Gorenstein flat.
  \item For every $\Rop$-module $M$ that has a complete projective
    resolution and for all $i\in\ZZ$ there are isomorphisms of
    functors on the category of $R$-modules,
    \begin{equation*}
      \Ctor{i}{M}{-} \dis \Ttor{i}{M}{-}\:.
    \end{equation*} 
  \end{eqc}
\end{prp}

\begin{proof} 
  \proofofimp{i}{ii} Let $M$ be an $\Rop$-module that has a complete
  projective resolution. By assumption, $M$ has finite Gorenstein flat
  dimension, so there is an $n\ge 0$ such that $\Tor{i}{M}{E}=0$ holds
  for every injective $R$-module $E$ and all $i > n$; see
  \thmcite[4.14]{CFH-11}. The desired isomorphisms of functors now
  follow from \prpref{copure} and \thmcite[6.7]{CCLP}.

  \proofofimp{ii}{i} Let $M$ be a Gorenstein projective $\Rop$-module
  and let $T$ be a totally acyclic complex of projective
  $\Rop$-modules with $M \is \Coker{(T_1 \to T_0)}$.  By assumption
  there are isomorphisms
  \begin{equation*}
    \H[i]{\tp{T}{E}} \is \Ctor{i}{M}{E}=0
  \end{equation*}
  for every injective $R$-module $E$ and all $i\in\ZZ$; see
  \pgref{ipg:vanishing} and \ref{ipg:gproj}. Thus $\tp{T}{E}$ is
  acyclic for every injective $R$-module $E$, and hence $M$ is
  Gorenstein flat.
\end{proof}

Again, there is an alternate proof that does not require any knowledge
of how stable homology compares to Tate homology.

\begin{proof}[Alternate proof of \pgref{prp:complete-Tate}]
  In the proof above, only the implication \proofofimp[]{i}{ii}
  references stable homology. Let $M$ be an $\Rop$-module that has a
  complete projective resolution \mbox{$T \to P \to M$}. Assuming that
  every Gorenstein projective $\Rop$-module is Gorenstein flat, it
  follows that the complex $\tp{T}{E}$ is acyclic for every injective
  $R$-module $E$; see \thmcite[2.2]{IEm12}. Thus one has
  $\Ttor{i}{M}{E} = \H[i]{\tp{T}{E}}=0$ for all $i\in\ZZ$, and
  \ref{ipg:reprove} finishes the argument.
\end{proof}

\begin{rmk}
  \label{rmk:tate}
  The similarity of \prpref{complete-Tate} to \thmcite[6.7]{CCLP}
  means that for a ring $R$ the following conditions are equivalent:
  \begin{eqc}
  \item For every $\Rop$-module $M$ with a complete projective
    resolution there are isomorphisms of functors $\Ctor{i}{M}{-} \is
    \Ttor{i}{M}{-}$ for all $i\in\ZZ$.
  \item For every $\Rop$-module $M$ with a complete projective
    resolution there are isomorphisms of functors $\Stor{i}{M}{-} \is
    \Ttor{i}{M}{-}$ for all $i\in\ZZ$.
  \end{eqc}
\end{rmk}



As a contrast to \ref{rmk:tate}, we make two remarks to clarify that
stable homology agrees with complete homology in cases where Tate
homology is not defined.

\begin{rmk}
  \label{exa:gor}
  Over a commutative artinian ring $R$ that is not Gorenstein, there
  exist finitely generated modules that do not have a complete
  projective resolution. For example, let $\m$ be a maximal ideal of
  $R$ such that the local ring $R_\m$ is not Gorenstein, then $R/\m$
  is such a module; see \prpcite[2.17 and thm.~2.19]{CFH-11}. In this
  case, Tate homology ``$\Ttor{}{R/\m}{-}$'' is not defined, but by
  the Main Theorem, complete homology $\Ctor{}{R/\m}{N}$ still agrees
  with stable homology $\Stor{}{R/\m}{N}$ for all finitely generated
  $R$-modules $N$.
\end{rmk}

The same phenomenon can occur over a ring with finite sfli;
cf.~\corref{sfli}.

\begin{rmk}
  \label{exa:vNr}
  Let $R$ be von Neumann regular, then every $\Rop$-module is flat. It
  follows that stable homology $\Stor{}{M}{-}$ and complete homology
  $\Ctor{}{M}{-}$ agree and vanish for every $\Rop$-module $M$; see
  \pgref{ipg:vanishing} and \corref{sfli}. It is, however, possible
  that Tate homology ``$\Ttor{}{M}{-}$'' is not defined.  To see this,
  let $R$ be a commutative von Neumann regular ring over which there
  are modules of infinite projective dimension; the existence of such
  rings is proved by Osofsky \cite[3.1]{BLO70}. Let $G$ be an
  $R$-module with a complete projective resolution $T \to P \to G$. It
  follows from \prpcite[7.6]{LWCHHl15} that $T$ is contractible. Thus,
  $G$ has a projective syzygy and hence it has finite projective
  dimension. An $R$-module $M$ of infinite projective dimension,
  therefore, does not have a complete projective resolution, and Tate
  homology ``$\Ttor{}{M}{-}$'' is not defined.
\end{rmk}

Over a ring $R$ with finite $\sfli R$, every Gorenstein projective
$\Rop$-module is Gorenstein flat; cf.~\ref{ipg:gflat}. The situation
is the same if $R$ is commutative and artinian, and more generally if
$R$ is right-perfect: Let $T$ be a totally acyclic complex of
projective $\Rop$-modules, let $I$ be an injective $R$-module and $E$
be a faithfully injective $\kk$-module. The complex $\tp{T}{I}$ is
acyclic if and only if the dual complex $\Hom[\kk]{\tp{T}{I}}{E} \is
\Hom[\Rop]{T}{\Hom[\kk]{I}{E}}$ is acyclic. The $\Rop$-module
$\Hom[\kk]{I}{E}$ is flat and, by the assumption on $R$, that means
projective. By \ref{ipg:gproj} and the isomorphism above, the
complex $\Hom[\kk]{\tp{T}{I}}{E}$ is acyclic, whence $\tp{T}{I}$
is acyclic.

The general question of whether or not Gorenstein projective modules
are Gorenstein flat remains open.  We have no example of a ring over
which complete homology agrees with stable homology, and over which
Gorenstein projective modules are not known to be Gorenstein flat.

\appendix
\section*{Appendix A}
\stepcounter{section}

\noindent
Let $M$ be an $\Rop$-module and let $N$ be and $R$-module. We show
that the natural comparison maps
$\mapdef{\s_i}{\Stor{i}{M}{N}}{\Ctor{i}{M}{N}}$ described in
\pgref{ipg:tes} are surjective.

\begin{ipg}
  Let $P\qra M$ be a projective resolution and $N\qra I$ be an
  injective resolution. The complex $\tp{P}{I}$ is the direct sum
  totalization---and $\btp{P}{I}$ is the product totalization---of the
  double complex $D = (D_m^n)_{m,n\ge 0}$ with $D_m^n = \tp{P_m}{I^n}$,
  anti-commuting squares, and exact rows:
  \begin{equation*}
    \xymatrix@=1.25pc{
      \vdots \ar[d] &
      \vdots \ar[d] &
      \vdots \ar[d] &
      &
      \vdots \ar[d] &
      \\
      \tp[]{P_m}{I^0} \ar[r] \ar[d] &
      \tp[]{P_m}{I^1} \ar[r] \ar[d] &
      \tp[]{P_m}{I^2} \ar[r] \ar[d] &
      \ \ \cdots \ \, \ar[r]&
      \tp[]{P_m}{I^n} \ar[r] \ar[d] &
      \cdots \\
      *{\vdots_{{}_{}}^{}} \ar[d] &
      *{\vdots_{{}_{}}^{}} \ar[d] &
      *{\vdots_{{}_{}}^{}} \ar[d] &
      &
      *{\vdots_{{}_{}}^{}} \ar[d] &
      \\
      \tp[]{P_1}{I^0} \ar[r] \ar[d] &
      \tp[]{P_1}{I^1} \ar[r] \ar[d] &
      \tp[]{P_1}{I^2} \ar[r] \ar[d] &
      \ \ \cdots \ \, \ar[r] &
      \tp[]{P_1}{I^n} \ar[r] \ar[d] &
      \cdots \\
      \tp[]{P_0}{I^0} \ar[r] &
      \tp[]{P_0}{I^1} \ar[r] &
      \tp[]{P_0}{I^2} \ar[r] &
      \ \ \cdots \ \, \ar[r] &
      \tp[]{P_0}{I^n} \ar[r] &
      \cdots
    }
  \end{equation*}
  For every $k > 0$ the complex $\tp{P}{I^{\ge k}}$ is the direct sum
  totalization of the double complex obtained by removing the columns
  $0, \ldots, k-1$ from $D$.

  Let $k\ge 0$; an element in $\tpp{P}{I^{\ge k}}_i$ is a
  sequence $v = (v^n)_{n\ge k}$ with $v^n$ in $D^n_{i+n}$,
  and $v^n=0$ for $n \gg 0$. Notice that one has $v^n=0$ for $n < -
  i$. That is, $v$ is supported at $D^n_{i+n}$ for finitely many $n
  \ge \max\{k,-i\}$.
\end{ipg}

We make repeated use of the technique of compressing cycles in tensor
products.  Though it is standard, we develop it here in the notation
of the double complex $D$.

\begin{ipg}
  \label{claim}
  Fix $i\in\ZZ$ and $k\ge 0$. Let $v$ be an element in $\tpp{P}{I^{\ge
      k}}_i$ supported at $D^n_{i+n}$ for $h\le n \le j$; its
  boundary $\partial(v)$ in $\tpp{P}{I^{\ge k}}_{i-1}$ is then
  supported at $D^{n}_{i-1+n}$ for $h \le n \le j+1$, and the
  following hold:
  \begin{rqm}
  \item[(a)] If $\partial(v)$ is supported at $D^{n}_{i-1+n}$ for $h
    \le n \le j'$ with $h < j' \le j$, then there exists an element
    $u$ in $\tpp{P}{I^{\ge k}}_{i+1}$ supported at $D^n_{i+1+n}$ for
    $j'-1\le n \le j-1$, such that $v - \partial(u)$ is supported at
    $D^n_{i+n}$ for $h \le n \le j'-1$.
  \item[(b)] If $\partial(v) = 0$, then for every $e$ with $\max\{k,-i\}
    \le e \le h$ here exists an element $u_e$ in $\tpp{P}{I^{\ge
        k}}_{i+1}$ supported at $D^n_{i+1+n}$ for $e \le n \le j-1$,
    such that $v - \partial(u_e)$ is supported at $D^e_{i+e}$.
  \end{rqm}
  Indeed, start at the last component of $v$, that is, at $v^j\in
  D^j_{i+j}$. That element is a cycle in row $i+j$ by the assumption
  $j'\le j$, so there is an element $u^{j-1}$ in $D_{i+j}^{j-1}$ with
  $\partial^h(u^{j-1}) = v^j$. Thus $v' = v - \partial(u^{j-1})$ is
  supported at $D^n_{i+n}$ for $h \le n \le j-1$. If $j' = j$, then
  $u^{j-1}$ is the element claimed in (a). If not, notice that one has
  $\partial(v') = \partial(v)$ and repeat; after $j-j'+1$ iterations
  one has the desired element $u$.

  Now, if $v$ is a cycle, then $\partial(v)$ has empty support. Thus
  the procedure above can be repeated until $v-\partial(u)$ is
  supported at $D^h_{i+h}$. From there it can be repeated to produce a
  $u$ with $v-\partial(u)$ supported at $D^e_{i+e}$ as long as
  $D^e_{i+e}$ is within the boundaries of the (truncated) double complex.
\end{ipg}

We can now describe the elements in $\Ctor{i}{M}{N} \is \lim_{k\ge
  0}\H[i]{\tp{P}{\tha{k}{I}}}$ in such a way that surjectivity of the
comparison map becomes almost trivial.

\begin{ipg}
  \label{surjective}
  Fix $i\in\ZZ$ and set $d = \max\{0,-i\}$; one has
  \begin{equation*}
    \lim_{k\ge 0}\H[i]{\tp{P}{\tha{k}{I}}} \is 
    \lim_{k\ge d}\H[i]{\tp{P}{\tha{k}{I}}}\:.
  \end{equation*}
  An element in $\lim_{k\ge d}\H[i]{\tp{P}{\tha{k}{I}}}$ is a
  compatible family of homology classes $([w^{\ge k}])_{k\ge d}$,
  where $w^{\ge k}$ is a cycle of degree $i$ in $\tp{P}{I^{\ge k}}$;
  by \pgref{claim}(b) we may assume that $w^{\ge k}$ is supported at
  the left edge of the truncated double complex, i.e.\ at
  $D^{k}_{i+k}$. Compatibility means that for each $k\ge d$ the cycle
  $w^{\ge k} - w^{\ge k+1}$ in $\tp{P}{I^{\ge k}}$ is a boundary:
  $w^{\ge k} - w^{\ge k+1} = \partial(x)$ for some $x$ in
  $\tp{P}{I^{\ge k}}$. It is supported at $D^k_{i+k}$ and
  $D^{k+1}_{i+1+k}$, so by \pgref{claim}(a) there exists an element
  $v^k\in D^{k}_{i+1+k}$ with $\partial(v^k) = w^{\ge k} - w^{\ge
    k+1}$. Observe that one has $\partial^v(v^k) = w^{\ge k}$ and
  $\partial^h(v^k) = -w^{\ge k+1}$.

  To see that $\s_i$ is surjective, let an element $([w^{\ge k}])_{k\ge d}$ in
  $\lim_{k\ge d}\H[i]{\tp{P}{\tha{k}{I}}}$ be given. Choose for each
  $k\ge d$ an element $v^k$ as above. The family $v = (v^k)_{k \ge d}$
  is an element in $(\btp{P}{I})_{i+1}$ and $\partial(v) = w^{\ge d}$ is in
  $\tp{P}{I}$; the equivalence class $[v]$ in $(\ttp{P}{I})_{i+1}$ is a
  cycle with $[\partial(v^{\ge k})] = [w^{\ge k}]$. Thus one has
  $\s_i([v]) = ([w^{\ge k}])_{k\ge d}$.
\end{ipg}

Next we briefly discuss injectivity.

\begin{ipg}
  \label{injective}
  Adopt the notation from \pgref{surjective}. Consider a cycle in
  $(\ttp{P}{I})_{i+1}$, represented by $z = (z^k)_{k\ge d} $ in
  $(\btp{P}{I})_{i+1}$, and assume $\s_i([z]) = 0$. To prove that $[z]$
  is zero in $\H[i+1]{\ttp{P}{I}} = \Stor{i}{M}{N}$ amounts to proving
  the existence of an element $v$ in $(\btp{P}{I})_{i+2}$ such that $z
  - \partial(v)$ has finite support, i.e.\ belongs to
  $\tpp{P}{I}_{i+1}$.  

  Consider $\s_i([z]) = ([\partial(z^{\ge k})])_{k \ge d}$ in
  $\lim_{k\ge d}\H[i]{\tp{P}{\tha{k}{I}}}$. As $\partial(z)$ is in
  $(\tp{P}{I})_{i}$, it follows that $\partial(z^{\ge k})$ for $k \gg
  0$ is supported at only one component of the double complex. That
  is, one can choose $k' > d$ such that for all $k\ge k'$ the element
  $b^k = \partial(z^{\ge k})$ in $(\tp{P}{I^{\ge k}})_i$ is supported
  at $D^{k}_{i+k}$; in fact, it is $\partial^v(z^k) =
  -\partial^h(z^{k-1})$.

  The assumption $\s_i([z]) =0$ implies, in particular, that for
  every $k\ge k'$ there exists an element $x^k$ in $(\tp{P}{I^{\ge
      k}})_{i+1}$ with $\partial(x^k) = b^k$. By \pgref{claim}(a) we may
  assume that $x^k$ is supported at $D^{k}_{i+k+1}$ where also $z^k$
  resides. Thus one has $\partial^v(x^k) = b^k = \partial^v(z^k)$ and
  $\partial^h(x^k) = 0$.
\end{ipg}

We have not been able to verify that the comparison map is injective,
i.e.\ an isomorphism, in situations that are not covered by
\prpref{copure}. In the primitive case of that result, $M$ is a module
with $\Tor{>0}{M}{E}=0$ for all injective $R$-modules $E$, so the
columns of $D$ are exact, and a straightforward diagram chase then
produces the element $v$ sought after in \pgref{injective}. As the
index $i$ and the cycle $z$ are arbitrary, this shows that the
comparison map is injective. However, we already knew that and more
from \prpref{copure}, so to conclude we ask:

\begin{bfhpg}[Question]
\label{question}
Let $M$ be an $\Rop$-module and $i$ be an integer. If for every
$R$-module $N$ there is an isomorphism $\Stor{i}{M}{N} \is
\Ctor{i}{M}{N}$, is then the comparison map $\s_i$ an isomorphism?
\end{bfhpg}

\section*{Appendix B}
\stepcounter{section}

\noindent
Kropholler states in \seccite[3.3]{PHK95} that the three
generalizations of Tate cohomology due to Benson and Carlson
\cite{DJBJFC92}, Mislin \cite{GMs94}, and P.~Vogel \cite{FGc92} are
isomorphic. A proof of this claim can be engineered by mimicking
arguments in Nucinkis's paper \cite{BEN98}. The purpose of this
appendix is to sketch how this would work.

First, a word on notation.

\begin{bfhpg}[Notation]
  For $R$-modules $M$ and $N$ we use, like in \dfnref{complete}, the
  symbols $\Text{i}{M}{N}$ for $i\in\ZZ$ to denote the P-completion of
  the cohomological functor $\Ext{i}{M}{-}$ evaluated at $N$. This is
  the standard notation for Tate cohomology, but no ambiguity arises
  as complete and Tate cohomology agree whenever the latter is
  defined; see \pgref{ipg:gproj} and Cornick and
  Kropholler~\thmcite[1.2]{JCrPHK97}.
\end{bfhpg}

\begin{ipg}
  With extra notation to distinguish them, the three families of
  cohomology modules are defined as follows,
  \begin{align*}
    {\text{(complete, Mislin)}} \hspace*{3em} \xExthat{comp.}{i}{M}{N}
    =& \colim_{k \ge 0}
    \sa_{k}\Ext{k+i}{M}{N}\:,\\
    {\text{(stable, Vogel)}} \hspace*{3em} \xExthat{stable}{i}{M}{N}
    =& \
    \H[-i]{\Hom{P}{Q}/\oHom{P}{Q}}\:,\text{ and}\\
    {\text{(Benson and Carlson)}} \hspace*{3em} \xExthat{B \&
      C}{i}{M}{N} =& \colim_{k \ge 0}
    \uHom{\syz{k}{M}}{\syz{k-i}{N}}\:.
  \end{align*}
  In the first line, $\sa_k$ denotes the $k^\mathrm{th}$ left
  satellite; see \chpcite[III]{careil}. In the second line, $P \qra M$
  and $Q \qra N$ are projective resolutions, and $\oHom{P}{Q}$ is the
  subcomplex of $\Hom{P}{Q}$ with modules $\oHom{P}{Q}_i =
  \coprod_{n\in\ZZ}\Hom{P_n}{Q_{n+i}}$; here the coproduct replaces
  the product used for the Hom functor.  In the third line,
  $\underline{\mathrm{Hom}}$ is notation for stable Hom: the
  $\kk$-module of homomorphisms modulo those that factor through a
  projective $R$-module. Finally, $\syz{n}{M}$ is the $n^\mathrm{th}$
  syzygy, $\Coker{(P_{n+1} \to P_{n})}$, in a projective resolution,
  $P \qra M$; by Schanuel's lemma it is unique up to a projective
  summand.
\end{ipg}

Mislin \seccite[4]{GMs94} proved that his complete cohomology theory
is isomorphic to Benson and Carlson's $\smash{\XExthat{B\&C}}$. Thus
it is sufficient to show that stable cohomology is isomorphic to
$\XExthat{B\&C}$. Our goal is to construct a morphism
\begin{equation*}
  \xExthat{stable}{}{M}{-} \xra{\quad\mu\quad} \xExthat{B \&
    C}{}{M}{-}
\end{equation*}
of cohomological functors, such that the maps
$\mapdef{\mu^N_i}{\xExthat{stable}{i}{M}{N}}{\xExthat{B\&C}{i}{M}{N}}$
are isomorphisms and compatible with the connecting homomorphisms.

\begin{bfhpg*}[Construction of $\mathbf{\mu}$]
  Fix an $i\in\ZZ$ and an $R$-module $N$. An element
  \begin{equation*}
    \widehat{\varphi} \in \xExthat{stable}{i}{M}{N} =
    \H[-i]{\Hom{P}{Q}/\oHom{P}{Q}}
  \end{equation*}
  is represented by a homomorphism $\varphi$ of homological degree
  $-i$, which is a chain map in high degrees, i.e., for $k\gg 0$ the
  diagram
  \begin{equation*}
    \xymatrix{\cdots \ar[r] &P_{k+1} \ar[r] \ar[d]^-{\f_{k+1}} & P_{k}
      \ar[r] \ar[d]^-{\f_k} & \syz{k}{M} \ar[d]^-{\widetilde{\f}_k}\\
      \cdots \ar[r] &Q_{k+1-i} \ar[r] & Q_{k-i} \ar[r] & \syz{k-i}{N}
    }
  \end{equation*}
  is commutative up to a sign $(-1)^i$.  The right-hand square defines
  $\widetilde{\varphi}_{k}$.  In this way, $\varphi$ defines element
  $\widetilde{\varphi}$ in $\xExthat{B \& C}{n}{M}{N}$. To see that
  this yields a homomorphism
  \begin{equation*}
    \dmapdef{\mu^N_i}{\xExthat{stable}{i}{M}{N}}{\xExthat{B \&
        C}{i}{M}{N}}\:, 
  \end{equation*}
  it must be verified that $\widetilde{\f}$ is independent of the
  choice of representative $\varphi$ of $\widehat{\varphi}$ in
  $\xExthat{stable}{i}{M}{N}$. If $\widehat{\f} = \widehat{\psi}$ in
  $\xExthat{stable}{i}{M}{N}$, then $\f- \psi$ is $0$-homotopic in
  high degrees, so for every $k\gg 0$ the induced homomorphism
  $\mapdef{\widetilde{\f}_k -
    \widetilde{\psi}_k}{\syz{k}{M}}{\syz{k-i}{N}}$ factors through the
  projective module $Q_{k-i}$, whence it is zero in
  $\uHom{\syz{k}{M}}{\syz{k-i}{N}}$.

  It is straightforward to verify that $\mu_i^N$ as defined above is
  natural in $N$ and that the family $\mu^N$ is compatible with the
  connecting homomorphisms.
\end{bfhpg*}

\begin{bfhpg*}[Injectivity]
  Let $\widehat{\f}$ be an element in the kernel of $\mu^N_i$. For
  every $k \gg 0$ the induced morphism
  $\mapdef{\widetilde{\f}_k}{\syz{k}{M}}{\syz{k-i}{N}}$ factors
  through a projective $R$-module $L$ and further through the
  surjection $Q_{k-i} \onto \syz{k-i}{N}$,
  \begin{equation*}
    \xymatrix@R=1.5pc{
      \cdots \ar[r] & P_{k+1} \ar[r] \ar[dd]_-{\f_{k+1}} & P_{k}
      \ar[rr] \ar[dd]_-{\f_k} \ar@{.>}[ddl]_-{\sigma_k} && 
      \syz{k}{M} \ar[dd]_-{\widetilde{\f}_k}
      \ar[dl] \ar@{.>}@/_{1.5pc}/[ddll]_-{\sigma_{k-1}}\\
      & & & L \ar[dr] \ar@{.>}[dl] \\
      \cdots \ar[r] &Q_{k+1-i} \ar[r] & Q_{k-i} \ar[rr] && \syz{k-i}{N}\:.
    }
  \end{equation*}
  It is now straightforward to construct the homotopies
  $\mapdef{\sigma_n}{P_n}{Q_{n+1-i}}$ for $n \ge k$, so $\f$ is
  $0$-homotopic in high degrees, i.e.\ one has $\hat{\f}=0$ in
  $\xExthat{stable}{i}{M}{N}$.
\end{bfhpg*}

\begin{bfhpg*}[Surjectivity] 
  An element $\widetilde{\f}$ in $\xExthat{B \& C}{i}{M}{N}$ is a
  family of elements in the direct system of modules
  $\uHom{\syz{k}{M}}{\syz{k-i}{N}}$.  Such a family is determined by
  an element $\underline{\f}$ in $\uHom{\syz{k}{M}}{\syz{k-i}{N}}$ for
  some $k \gg 0$, and $\underline{\f}$ is represented by a
  homomorphism $\f \in \Hom{\syz{k}{M}}{\syz{k-i}{N}}$. Lifting $\f$
  to a morphism of projective resolutions $\Thb{k}{P} \to
  \susp^i\Thb{k-i}{Q}$ yields an element $\widehat{\f}$ in
  $\xExthat{stable}{i}{M}{N}$ with $\mu^N_i(\widehat{\f}) =
  \widetilde{\f}$.
\end{bfhpg*}

The next result is used in the proof of \thmref{duality}. It is dual
to \ref{ipg:lim}\eqref{2} and has a similar proof, which can also be
extracted from \seccite[4]{GMs94}.

\begin{lem}
  \protect\pushQED{\qed}%
  \label{lem:complete cohomology}
  Let $M$ and $N$ be $R$-modules with projective resolutions $P \qra
  M$ and $Q \qra N$. For every $i\in\ZZ$ there is an isomorphism of
  $\kk$-modules
  \begin{equation*}
    \Text{i}{M}{N} \dis \colim_{k\ge 0}\HH[i]{\Hom{P}{\Thb{k}{Q}}}
  \end{equation*}
  which is natural in the second argument.
\end{lem}

\section*{Acknowledgments}

\noindent
We thank Alex Martsinkovsky and Guido Mislin for making the theses
\cite{JRs-phd,MTr-phd} of Jeremy Russell and Mauro Triulzi available
to us, and we extend our gratitude to the referee for pertinent
comments that improved the presentation at several points and led to
Question \pgref{question}.

\def\cprime{$'$}
  \providecommand{\arxiv}[2][AC]{\mbox{\href{http://arxiv.org/abs/#2}{\sf
  arXiv:#2 [math.#1]}}}
  \providecommand{\oldarxiv}[2][AC]{\mbox{\href{http://arxiv.org/abs/math/#2}{\sf
  arXiv:math/#2
  [math.#1]}}}\providecommand{\MR}[1]{\mbox{\href{http://www.ams.org/mathscinet-getitem?mr=#1}{#1}}}
  \renewcommand{\MR}[1]{\mbox{\href{http://www.ams.org/mathscinet-getitem?mr=#1}{#1}}}
\providecommand{\bysame}{\leavevmode\hbox to3em{\hrulefill}\thinspace}
\providecommand{\MR}{\relax\ifhmode\unskip\space\fi MR }
\providecommand{\MRhref}[2]{%
  \href{http://www.ams.org/mathscinet-getitem?mr=#1}{#2}
}
\providecommand{\href}[2]{#2}

\end{document}